\documentclass[11pt,a4paper]{amsart}
\usepackage{amssymb}
\usepackage{mathabx}
\usepackage{mathrsfs}
\usepackage{syntonly}
\usepackage{amsmath}
\usepackage{amsthm}
\usepackage{amsfonts}
\usepackage{amssymb}
\usepackage{latexsym}
\usepackage{amscd,amssymb,amsopn,amsmath,amsthm,graphics,amsfonts,mathrsfs,accents,enumerate,verbatim,calc}
\usepackage[dvips]{graphicx}
\usepackage[colorlinks=true,linkcolor=red,citecolor=blue]{hyperref}
\usepackage[all]{xy}

\date{}
\allowdisplaybreaks[4] \footskip=15pt
\renewcommand{\uppercasenonmath}[1]{}

\numberwithin{equation}{section} \theoremstyle{plain}
\newtheorem{lem}{Lemma}[section]
\newtheorem{cor}[lem]{Corollary}
\newtheorem{prop}[lem]{Proposition}
\newtheorem{thm}[lem]{Theorem}

\numberwithin{equation}{section} \theoremstyle{definition}

\newtheorem{definition}[lem]{Definition}
\newtheorem{Ex}[lem]{Example}

\newtheorem{rem}[lem]{Remark}

\newenvironment{rmk}{\begin{rem}\rm}{\end{rem}}
\newenvironment{df}{\begin{definition}\rm}{\end{definition}}
\newenvironment{ex}{\begin{Ex}\rm}{\end{Ex}}

\begin{document}
\begin{center}
{\large  \bf  Compatibility of HRS tilt and completion of $t$-structures in triangulated categories}

\vspace{0.5cm}  Jiaojiao Lu\footnote{Corresponding author}, Zhongkui Liu, Renyu Zhao
\end{center}

\bigskip
\centerline { \bf  Abstract}
\medskip

\leftskip10truemm \rightskip10truemm \noindent Let $K$ be a field, and $\mathcal{T}$ a $K$-linear essentially small triangulated category equipped with an extendable $t$-structure $(\mathcal{T}^{\leq0}, \mathcal{T}^{\geq0})$ with respect to a good metric $\mathfrak{B}$. Given a torsion class $\mathcal{X}$ in the heart of $(\mathcal{T}^{\leq0}, \mathcal{T}^{\geq0})$, we prove that lifting the HRS-tilt of $(\mathcal{T}^{\leq0}, \mathcal{T}^{\geq0})$ at $\mathcal{X}$ along the completion of $\mathcal{T}$ coincides with the HRS-tilt of the lifted $t$-structure at the completion of $\mathcal{X}$. As an application, we provide the compatibility result between left silting mutation in $\mathcal{T}$ and HRS tilting in the ambient completion.    \\[2mm]

\leftskip0truemm \rightskip0truemm
\section { \bf Introduction}
The notion of a $t$-structure was introduced by Beilinson, Bernstein and Deligne \cite{BBD} as a triangulated analogue of torsion pairs in abelian categories. Its heart provide a natural abelian subcategory of the ambient triangulated category. It has since become a fundamental and widely applicable tool in algebra, geometry, and topology; see \cite{B, HKM, KV, S}. The HRS-tilt $t$-structure, introduced by Happel, Reiten and Smal{\o} \cite{HRS}, is a tool that convert the torsion theory in abelian category into new abelian category by transforming $t$-structures. This construction generalizes  the classical tilting theory from module-based operations to a categorical framework using torsion pairs and $t$-structures, thereby serving as a key bridge between the representation theory of algebras and the algebraic geometry. Adachi, Mizuno and Yang \cite{AMY} established a relationship between the left mutation of a silting object and the semisimple tilt of the associated $t$-structure.

The concept of completion was introduced by Neeman in his pioneering work \cite{N1}, which provides a method to construct a new triangulated category using the internal structure of the original one. The procedure mainly relies on metrics that are a technical tool on triangulated categories. Biswas, Chen, Rahul, Parker and Zheng in their main theorem \cite[Theorem 3.5]{BCRPZ} proved that under suitable assumptions on the metric, the completion of a bounded $t$-structure is again a $t$-structure. They further showed that an essentially small category admitting a bounded $t$-structure coincides with its completion under an appropriate finite-dimensional assumption. For more on completions of triangulated categories see \cite{CG, HL, M}.

So starting from a given $t$-structure on a $K$-linear essentially small triangulated category $\mathcal{T}$ with respect to a good metric $\mathfrak{B}$, there are multiple ways to construct new $t$-structures. Let $(\mathcal{T}^{\leq0}, \mathcal{T}^{\geq0})$ be an extendable $t$-structure on $\mathcal{T}$ with respect to $\mathfrak{B}$. Then, on the one hand,
$(\mathfrak{G}(\mathcal{T}^{\leq0}), \mathfrak{G}(\mathcal{T}^{\geq0}))$ is a $t$-structure of the triangulated category $\mathfrak{G}(\mathcal{T})$, where $\mathfrak{G}$ is the completion functor with respect to the metric $\mathfrak{B}$. On the other hand, for a torsion class $\mathcal{X}$ of the heart $\mathcal{T}^{0}$ of $(\mathcal{T}^{\leq0}, \mathcal{T}^{\geq0})$,  define the following two full subcategories of $\mathcal{T}$:
\begin{center}
$\mu^{\mathrm{L}}_{\mathcal{X}}\mathcal{T}^{\leq0}:=\mathcal{T}^{\leq-1}*\mathcal{X},~
\mu^{\mathrm{L}}_{\mathcal{X}}\mathcal{T}^{\geq0}:=\Sigma\mathcal{X}^{\perp_{\mathcal{T}^{0}}}*\mathcal{T}^{\geq0}$.
\end{center}
 Then the pair $(\mu^{\mathrm{L}}_{\mathcal{X}}\mathcal{T}^{\leq0},\mu^{\mathrm{L}}_{\mathcal{X}}\mathcal{T}^{\geq0})$ is also a  $t$-structure on $\mathcal{T}$. We call this $t$-structure HRS-tilt of $(\mathcal{T}^{\leq0}, \mathcal{T}^{\geq0})$ with respect to  $\mathcal{X}$. Now we have the following diagram:
 \[ \xymatrix{
 \mathcal{T}:\ar[d]^{\mathfrak{G}} & (\mathcal{T}^{\leq0}, \mathcal{T}^{\geq0}) \ar[r]\ar[d]^{\mathfrak{G}} &(\mu^{\mathrm{L}}_{\mathcal{X}}\mathcal{T}^{\leq0},\mu^{\mathrm{L}}_{\mathcal{X}}\mathcal{T}^{\geq0})   \ar[d]^{\mathfrak{G}}  \\
 \mathfrak{G}(\mathcal{T}):  & (\mathfrak{G}(\mathcal{T}^{\leq0}), \mathfrak{G}(\mathcal{T}^{\geq0})) &  (\mathfrak{G}(\mu^{\mathrm{L}}_{\mathcal{X}}\mathcal{T}^{\leq0}), \mathfrak{G}(\mu^{\mathrm{L}}_{\mathcal{X}}\mathcal{T}^{\geq0}))
 .}\]
This leads to the following question:

$\mathbf{Question}$: What is the relationship between the lifted HRS-tilted $t$-structures $(\mathfrak{G}(\mu^{\mathrm{L}}_{\mathcal{X}}\mathcal{T}^{\leq0}), \mathfrak{G}(\mu^{\mathrm{L}}_{\mathcal{X}}\mathcal{T}^{\geq0}))$ and the HRS-tilt of the lifted $t$-structure $(\mathfrak{G}(\mathcal{T}^{\leq0}), \mathfrak{G}(\mathcal{T}^{\geq0}))$?

The purpose of this article is to show that the $t$-structure $(\mathfrak{G}(\mu^{\mathrm{L}}_{\mathcal{X}}\mathcal{T}^{\leq0}), \mathfrak{G}(\mu^{\mathrm{L}}_{\mathcal{X}}\mathcal{T}^{\geq0}))$ is exactly the HRS-tilt of $t$-structure $(\mathfrak{G}(\mathcal{T}^{\leq0}), \mathfrak{G}(\mathcal{T}^{\geq0}))$.
More precisely, we show the following theorem.

\begin{thm} \label{thmA} Let $\mathcal{T}$ be a $K$-linear essentially small triangulated category with a good metric $\mathfrak{B}=\{\mathfrak{B}_{n}\}_{n\in\mathbb{N}}$, $(\mathcal{T}^{\leq0},\mathcal{T}^{\geq0})$ be a $t$-structure on $\mathcal{T}$ and $\mathcal{X}$ be a torsion class of the heart $\mathcal{T}^{0}$. If $(\mathcal{T}^{\leq0},\mathcal{T}^{\geq0})$ is extendable with respect to the good metric $\mathfrak{B}$, then
 $$\mathfrak{G}(\mu^{\mathrm{L}}_{\mathcal{X}}\mathcal{T}^{\leq0})=\mu^{\mathrm{L}}_{\mathfrak{G}(\mathcal{X})}\mathfrak{G}(\mathcal{T}^{\leq0}), \ \ \
 \mathfrak{G}(\mu^{\mathrm{L}}_{\mathcal{X}}\mathcal{T}^{\geq0})=\mu^{\mathrm{L}}_{\mathfrak{G}(\mathcal{X})}\mathfrak{G}(\mathcal{T}^{\geq0}).$$
\end{thm}

Thus, we have the following commutative diagram:
\[ \xymatrix{
  (\mathcal{T}^{\leq0}, \mathcal{T}^{\geq0}) \ar[r]^{T_{\mathrm{HRS}}^{\mathcal{X}}}\ar[d]^{\mathfrak{G}} &(\mu^{\mathrm{L}}_{\mathcal{X}}\mathcal{T}^{\leq0},\mu^{\mathrm{L}}_{\mathcal{X}}\mathcal{T}^{\geq0})   \ar[d]^{\mathfrak{G}}  \\
  (\mathfrak{G}(\mathcal{T}^{\leq0}), \mathfrak{G}(\mathcal{T}^{\geq0}))\ar[r]^{T_{\mathrm{HRS}}^{\mathfrak{G}(\mathcal{X})}} &  (\mathfrak{G}(\mu^{\mathrm{L}}_{\mathcal{X}}\mathcal{T}^{\leq0}), \mathfrak{G}(\mu^{\mathrm{L}}_{\mathcal{X}}\mathcal{T}^{\geq0})).
 }\]

This paper is organized as follows. In Section 2, we collect some basic definitions and preliminary results concerning $t$-structures and completions on triangulated categories. Section 3 is devoted to verify Theorem \ref{thmA}. In Section 4, we present some applications of Theorem \ref{thmA}.
\section{\bf Preliminaries}
In this section, we set notations and recall preliminary materials on $t$-structures and completions of triangulated categories.

\subsection{General notations and conventions}
\noindent

Throughout the paper, let $K$ be a field. For an additive category $\mathcal{T}$, we say that $\mathcal{T}$ is \emph{$K$-linear} if for any objects $X,Y\in\mathcal{T}$, $\mathrm{Hom}_{\mathcal{T}}(X,Y)$ is a $K$-vector space and the composition operation of morphisms is $K$-bilinear. Let $\mathcal{T}$ be a $K$-linear additive category. Recall that $\mathcal{T}$ is \emph{Hom-finite} if all morphism spaces of $\mathcal{T}$ are finite-dimensional over $K$. We say that $\mathcal{T}$ is \emph{Krull-Schmidt} if every object $M$ of $\mathcal{T}$ is isomorphic to $ M_{1}\oplus \cdots\oplus M_{n}$ with $M_{1}, \ldots, M_{n}$ having local endomorphism algebras. For a full subcategory $\mathcal{X}$ of $\mathcal{T}$, define subcategories $\mathcal{X}^{\perp}:=\mathcal{X}^{\perp_{\mathcal{T}}}$ and ${^{\perp}\mathcal{X}}:={^{\perp_{\mathcal{T}}}\mathcal{X}}$ of $\mathcal{T}$ as
\begin{center}
$\mathcal{X}^{\perp_{\mathcal{T}}}:=\{T\in\mathcal{T}~|~\mathrm{Hom}_{\mathcal{T}}(X,T)=0,~\forall~X\in\mathcal{X}\}$,\\
${^{\perp_{\mathcal{T}}}\mathcal{X}}:=\{T\in\mathcal{T}~|~\mathrm{Hom}_{\mathcal{T}}(X,T)=0,~\forall~X\in\mathcal{X}\}$.
\end{center}

Assume that $\mathcal{T}$ is a triangulated category with the shift functor $\Sigma$. We say that $\mathcal{T}$ is \emph{essentially small} if the isomorphism classes of objects in $\mathcal{T}$ form a set. For two full subcategories $\mathcal{X}$ and $\mathcal{Y}$ of $\mathcal{T}$, we denote by $\mathcal{X}\ast\mathcal{Y}$ the \emph{extension} of $\mathcal{X}$ and $\mathcal{Y}$, that is the full subcategory of $\mathcal{T}$ consisting of objects $Z$ which admit a triangle $X\rightarrow Z\rightarrow Y\rightarrow \Sigma X$ in $\mathcal{T}$ with $X\in\mathcal{X}$ and $Y\in\mathcal{Y}$. If $\mathcal{X}\ast\mathcal{X}\subseteq\mathcal{X}$, then $\mathcal{X}$ is said to be \emph{closed under extensions} in $\mathcal{T}$. We define $\mathrm{Hom}_{\mathcal{T}}(\mathcal{X},\mathcal{Y})$ as the union of $\mathrm{Hom}_{\mathcal{T}}(X,Y)$ for all $X\in\mathcal{X}$ and $Y\in\mathcal{Y}$. For a morphism $f:~X\rightarrow Y$ in $\mathcal{T}$, the third term $Z$ in a triangle $X\xrightarrow{f} Y\rightarrow Z\rightarrow \Sigma X$ in $\mathcal{T}$ is called the \emph{cone} of $f$ and is denoted by $\mathrm{Cone}(f)$. We denote by $\mathrm{thick}\mathcal{X}$ the \emph{thick subcategory} of $\mathcal{T}$ generated by $\mathcal{X}$, that is the smallest triangulated subcategory of $\mathcal{T}$ containing $\mathcal{X}$ and closed under taking direct summands. For more knowledge on triangulated categories, we refer the reader to \cite{H, N0}.

\subsection{$t$-structures on triangulated categories}
\noindent

We collect some definitions and facts from \cite{BBD}. Let $\mathcal{T}$ be a triangulated category with shift functor $\Sigma$.

\begin{df}
$($\cite[Definition 1.3.1]{BBD}$)$\label{dfn0} A pair $(\mathcal{T}^{\leq0},\mathcal{T}^{\geq0})$ of full subcategories of $\mathcal{T}$ is called a \emph{$t$-structure} on $\mathcal{T}$ if the following conditions are satisfied:

(t1) $\mathcal{T}^{\leq-1}\subseteq\mathcal{T}^{\leq0}$ and $\mathcal{T}^{\geq0}\subseteq\mathcal{T}^{\geq-1}$, where $\mathcal{T}^{\leq n}:=\Sigma^{-n}\mathcal{T}^{\leq0}$ and $\mathcal{T}^{\geq n}:=\Sigma^{-n}\mathcal{T}^{\geq0}$ for any $n\in\mathbb{Z}$;

(t2) $\mathrm{Hom}_{\mathcal{T}}(\mathcal{T}^{\leq-1},\mathcal{T}^{\geq0})=0$;

(t3) For each $X\in\mathcal{T}$, there exists a triangle $X^{'}\rightarrow X\rightarrow X^{''}\rightarrow \Sigma X^{'}$ in $\mathcal{T}$ with $X^{'}\in\mathcal{T}^{\leq-1}$ and $X^{''}\in\mathcal{T}^{\geq0}$ $($i.e. $\mathcal{T}=\mathcal{T}^{\leq-1}\ast\mathcal{T}^{\geq0})$.
\end{df}

The categories $\mathcal{T}^{\leq0}$ and $\mathcal{T}^{\geq0}$ are called the \emph{aisle} and the \emph{coaisle} of the $t$-structure $(\mathcal{T}^{\leq0},\mathcal{T}^{\geq0})$, respectively.  The category $\mathcal{T}^{0}:=\mathcal{T}^{\leq0}\bigcap\mathcal{T}^{\geq0}$ is called the \emph{heart} of the $t$-structure $(\mathcal{T}^{\leq0},\mathcal{T}^{\geq0})$, and $\mathcal{T}^{0}$ is an abelian category (see \cite[Theorem 1.3.6]{BBD}).

It is easy to see that for every integer $n$, the pair $(\mathcal{T}^{\leq n},\mathcal{T}^{\geq n})$ is also a $t$-structure with the heart $\mathcal{T}^{n}:=\mathcal{T}^{\leq n}\bigcap\mathcal{T}^{\geq n}$. Note that $\mathcal{T}^{\leq n}$ and $\mathcal{T}^{\geq n}$ are additive subcategories, which are closed under extensions and direct summands. By the conditions (t1) and (t2) in Definition \ref{dfn0}, for the $t$-structure $(\mathcal{T}^{\leq n},\mathcal{T}^{\geq n})$, we have
\begin{center}
$(\mathcal{T}^{\leq n})^{\perp}=\mathcal{T}^{\geq n+1}$ and ${^{\perp}(\mathcal{T}^{\geq n+1})}=\mathcal{T}^{\leq n}$.
\end{center}

Up to isomorphism, for each $X\in\mathcal{T}$, there exists a unique triangle $X^{\leq n-1}\rightarrow X\rightarrow X^{\geq n}\rightarrow \Sigma X^{\leq n-1}$ in $\mathcal{T}$ with $X^{\leq n-1}\in\mathcal{T}^{\leq n-1}$ and $X^{\geq n}\in\mathcal{T}^{\geq n}$. So the correspondences $X\mapsto X^{\leq n-1}$ and $X\mapsto X^{\geq n}$ extend to functors $\sigma^{\leq n-1}:\mathcal{T}\rightarrow\mathcal{T}^{\leq n-1}$ and $\sigma^{\geq n}:\mathcal{T}\rightarrow\mathcal{T}^{\geq n}$, respectively, which are called the \emph{truncation functors}. The functor $\sigma^{\leq n-1}$ is a right adjoint to the inclusion $\mathcal{T}^{\leq n-1}\rightarrow\mathcal{T}$, and the functor $\sigma^{\geq n}$ is a left adjoint to the inclusion $\mathcal{T}^{\geq n}\rightarrow\mathcal{T}$.

Let
\begin{center}
$\mathcal{T}^{-}:=\bigcup\limits_{n\in\mathbb{N}}\mathcal{T}^{\leq n}$, $\mathcal{T}^{+}:=\bigcup\limits_{n\in\mathbb{N}}\mathcal{T}^{\geq -n}$ and $\mathcal{T}^{b}:=\mathcal{T}^{-}\bigcap\mathcal{T}^{+}.$
\end{center}We say the $t$-structure $(\mathcal{T}^{\leq0},\mathcal{T}^{\geq0})$ is \emph{bounded above}, \emph{bounded below} or \emph{bounded} if $\mathcal{T}=\mathcal{T}^{-}, \mathcal{T}^{+} \mathrm{or}~\mathcal{T}^{b}$, respectively. We denote by $t$-$\mathrm{str}(\mathcal{T})$ the set of bounded $t$-structures on $\mathcal{T}$.

\subsection{Completions of triangulated categories}
\noindent

Let $\mathcal{T}$ be an essentially small triangulated category with shift functor $\Sigma$. In this subsection, we recall the definition of completions of triangulated categories. For details, we refer to \cite{N1, N2, BCRPZ}.

\begin{df}$($\cite[Definition 10]{N2}$)$ A sequence $\mathfrak{B}:=\{\mathfrak{B}_{n}\}_{n\in\mathbb{N}}$ of full subcategories of $\mathcal{T}$ is called a \emph{metric} on $\mathcal{T}$ if $\mathfrak{B}_{n}*\mathfrak{B}_{n}=\mathfrak{B}_{n}$ and $0\in\mathfrak{B}_{n}$ for all $n\in\mathbb{N}$.

Moreover, a metric $\mathfrak{B}=\{\mathfrak{B}_{n}\}_{n\in\mathbb{N}}$ is called a \emph{good metric} on $\mathcal{T}$ if $\Sigma^{-1}\mathfrak{B}_{n+1}\bigcup\mathfrak{B}_{n+1}\bigcup\Sigma\mathfrak{B}_{n+1}\subseteq\mathfrak{B}_{n}$ for all $n\in\mathbb{N}$.

The metric $\mathfrak{B}$ is said to be \emph{finer} than another metric $\mathfrak{N}:=\{\mathfrak{N}_{n}\}_{n\in\mathbb{N}}$ if for each $n$, there exists an $m\in\mathbb{N}$ such that $\mathfrak{B}_{m}\subseteq\mathfrak{N}_{n}$. We denote this by $\mathfrak{B}\leq\mathfrak{N}$. The metrics $\mathfrak{B}$ and $\mathfrak{N}$ are said to be \emph{equivalent} if $\mathfrak{B}\leq\mathfrak{N}\leq\mathfrak{B}$.
\end{df}

\begin{rmk}\label{rmk0} By the definition of good metric, we obtain $\mathfrak{B}_{n+|j|}\subseteq\Sigma^{j}\mathfrak{B}_{n}$ for any $j\in\mathbb{Z}$.
\end{rmk}

\begin{ex}\label{exa2} Let $\mathcal{X}$ be a full subcategory of $\mathcal{T}$ and $(\mathcal{T}^{\leq0},\mathcal{T}^{\geq0})$ a $t$-structure on $\mathcal{T}$. Put $\mathfrak{B}_{n}:=\mathcal{X}\bigcap\mathcal{T}^{\le-n}$ for any $n\in\mathbb{N}$. Then $\{\mathfrak{B}_{n}\}_{n\in\mathbb{N}}$ is a good metric on $\mathcal{X}$ and it is called \emph{the good metric induced from the aisle of $(\mathcal{T}^{\leq0},\mathcal{T}^{\geq0})$} by \cite[Definition 2.9]{BCRPZ}.
\end{ex}

\begin{df}$($\cite[Definition 1.6]{N1}$)$ Let $\mathfrak{B}$ be a good metric of $\mathcal{T}$. A chain of morphisms $\{X_{\bullet},f_{\bullet}\}:~X_{0}\xrightarrow{f_{1}}X_{1}\xrightarrow{f_{2}}X_{2}\xrightarrow{f_{3}}X_{3}\xrightarrow{}\cdots$ in $\mathcal{T}$ is called a \emph{Cauchy sequence} with respect to $\mathfrak{B}$ if for any $i\geq1$, there exists $n_{i}\geq1$ such that $\mathrm{Cone}(f_{j})\in\mathfrak{B}_{i}$ for all $j\geq n_{i}$.
\end{df}

\begin{rmk} Equivalent good metrics yield the same Cauchy sequence.
\end{rmk}

We denote by Mod-$\mathcal{T}$ the category of additive contravariant functions from $\mathcal{T}$ to the category of abelian groups, and let
\begin{center}
$Y:\mathcal{T}\rightarrow \text{Mod-}\mathcal{T}$
\end{center}the Yoneda embedding $T\mapsto\mathrm{Hom}_{\mathcal{T}}(-,T)$. We note that the Yoneda embedding $Y$ is fully faithful.

\begin{df}$($\cite[Definition 1.10]{N1}$)$\label{dfn1} Let $\mathfrak{B}$ be a good metric of $\mathcal{T}$, and $\mathcal{X}$ a full subcategory of $\mathcal{T}$. The full subcategories $\mathfrak{L}(\mathcal{X}),\mathfrak{C}(\mathcal{T})$ and $\mathfrak{G}(\mathcal{X})$ of Mod-$\mathcal{T}~($all of them with respect to $\mathfrak{B})$ are defined as follows:

$(1)$ The objects of  $\mathfrak{L}(\mathcal{X})$ are the functors $F$ in Mod-$\mathcal{T}$ such that $F\simeq \displaystyle\operatorname*{colim}_{n\to} Y(X_{n})$, where $\{X_{\bullet}\}$ is a Cauchy sequence in $\mathcal{T}$ with $X_{n}\in\mathcal{X}$ for all $n\in\mathbb{N}$.

$(2)$ The objects of  $\mathfrak{C}(\mathcal{T})$ are the functors $F$ in Mod-$\mathcal{T}$, they are given by the formula
\begin{center}
$\mathfrak{C}(\mathcal{T})=
\{F\in\text{Mod-}\mathcal{T}|\bigcap\limits_{j\in\mathbb{Z}}\bigcup\limits_{i\in\mathbb{N}}\mathrm{Hom}(Y(\Sigma^{j}\mathfrak{B}_{i}),F)\}$.
\end{center}
It implies that $F(\mathfrak{B}_{i})=0$ for some $i\geq0$.

$(3)$ $\mathfrak{G}(\mathcal{X}):=\mathfrak{L}(\mathcal{X})\bigcap\mathfrak{C}(\mathcal{T})$ and it is called the \emph{completion} of $\mathcal{T}$ with respect to $\mathfrak{B}$.
\end{df}

\begin{rmk} For a full subcategory $\mathcal{X}$ of $\mathcal{T}$, we have inclusion relation
\begin{center}
$Y(\mathcal{X})\subseteq\mathfrak{L}(\mathcal{X})$.
\end{center}
\end{rmk}

There is an autoequivalence
\begin{align*}
&\Sigma:\text{Mod-}\mathcal{T}\rightarrow\text{Mod-}\mathcal{T}\\
&\qquad{F\mapsto[\Sigma(F):T\mapsto F(\Sigma^{-1}T)]}
\end{align*}
for $F\in\text{Mod-}\mathcal{T}$ and $T\in\mathcal{T}$. It is easy to verify that $\Sigma^{j}(Y(T))\simeq Y(\Sigma^{j}T)$ for any $j\in\mathbb{Z}$.

\begin{lem}$($\cite[Theorem 2.11]{N1} $\mathrm{or}$ \cite[Theorem 15]{N2}$)$\label{thm1}
Let $\mathcal{T}$ be an essentially small triangulated category with a good metric $\mathfrak{B}$. Then the category $\mathfrak{G}(\mathcal{T})$ with its autoequivalence $\Sigma$ is triangulated, where the triangles are given by the sequences $A\xrightarrow{\alpha} B\xrightarrow{\beta} C\xrightarrow{\gamma}\Sigma A$ in $\mathfrak{G}(\mathcal{T})$ which are isomorphic to the colimit of the image under Yoneda functor $Y$ of a Cauchy sequence (with respect to $\mathfrak{B}$) of triangles in $\mathcal{T}$:
$$\{A_{\bullet}\}\xrightarrow{\alpha_{\bullet}}\{B_{\bullet}\}\xrightarrow{\beta_{\bullet}}\{C_{\bullet}\}\xrightarrow{\gamma_{\bullet}}\Sigma\{A_{\bullet}\}.$$
\end{lem}

\begin{lem}$($\cite[Lemma 2.14]{BCRPZ}$)$
$(1)$ For any full subcategory $\mathcal{X}$ of $\mathcal{T}$ and for any $j\in\mathbb{Z}$, there are equalities of additive categories:
\begin{center}
$\mathfrak{L}(\Sigma^{j}\mathcal{X})=\Sigma^{j}\mathfrak{L}(\mathcal{X})$, $\mathfrak{C}(\Sigma^{j}\mathcal{T})=\Sigma^{j}\mathfrak{C}(\mathcal{T})$ and $\mathfrak{G}(\Sigma^{j}\mathcal{X})=\Sigma^{j}\mathfrak{G}(\mathcal{X})$.
\end{center}
$(2)$ The functor $\Sigma$ restricts to automorphisms of additive categories $\mathfrak{L}(\mathcal{X}),\mathfrak{C}(\mathcal{T})$ and $\mathfrak{G}(\mathcal{X})$.\\
$(3)$ Equivalent good metrics of $\mathcal{T}$ yield the same $\mathfrak{L}(\mathcal{X}),\mathfrak{C}(\mathcal{T})$ and $\mathfrak{G}(\mathcal{X})$.
\end{lem}

\begin{df}$($\cite[Definition 3.2]{BCRPZ}$)$ Let $\mathfrak{B}$ be a metric of $\mathcal{T}$. A $t$-structure $(\mathcal{T}^{\leq0},\mathcal{T}^{\geq0})$ on $\mathcal{T}$ is called \emph{extendable} with respect to the metric $\mathfrak{B}$ if there exists a natural number $n$ such that $\mathfrak{B}_{n}\subseteq\mathcal{T}^{\leq0}$.
\end{df}
We note that for an extendable $t$-structure $(\mathcal{T}^{\leq0},\mathcal{T}^{\geq0})$, the metric $\mathfrak{B}$ is finer than the good metric $\{\mathcal{T}^{\leq-n}\}_{n\in\mathbb{N}}$ on $\mathcal{T}$.

\begin{lem}$($\cite[Theorem 3.5]{BCRPZ}$)$\label{thm2} Let $\mathfrak{B}$ be a good metric of $\mathcal{T}$, and $(\mathcal{T}^{\leq0},\mathcal{T}^{\geq0})$ be an extendable $t$-structure on $\mathcal{T}$ with respect to $\mathfrak{B}$. Then the following statements hold:

$(1)$ The pair $(\mathfrak{G}(\mathcal{T}^{\leq0}),\mathfrak{G}(\mathcal{T}^{\geq0}))$ is a $t$-structure on $\mathfrak{G}(\mathcal{T})$ with the heart $Y(\mathcal{T}^{0})$.

$(2)$ Let $\mathcal{X}$ be any full subcategory of $\mathcal{T}^{\geq0}$. Then the restriction of the Yoneda functor $Y$ to $\mathcal{X}$ yields an equivalence $Y|_{\mathcal{X}}:\mathcal{X}\rightarrow\mathfrak{G}(\mathcal{X})$ of additive categories.

$(3)$ If the $t$-structure $(\mathcal{T}^{\leq0},\mathcal{T}^{\geq0})$ is bounded above, then so is the $t$-structure $(\mathfrak{G}(\mathcal{T}^{\leq0}),\mathfrak{G}(\mathcal{T}^{\geq0}))$.
\end{lem}

\subsection{HRS-tilt of $t$-structures}
\noindent

Let $\mathcal{T}$ be a $K$-linear essentially small triangulated category with shift functor $\Sigma$. In this subsection, we recall the notion of HRS-tilt $t$-structures on triangulated categories.

\begin{df}$($\cite{D}$)$ Let $\mathcal{A}$ be a $K$-linear abelian category, $\mathcal{X}$ be a full subcategory of $\mathcal{A}$. We call $\mathcal{X}$ a \emph{torsion class} if, for each $a\in\mathcal{A}$, there exists an exact sequence $0\rightarrow x\rightarrow a\rightarrow y\rightarrow0$ with $x\in\mathcal{X}$ and $y\in\mathcal{X}^{\perp}$.

We denote by $\mathrm{tors}(\mathcal{A})$ the set of torsion classes of $\mathcal{A}$.
\end{df}

For a $t$-structure $(\mathcal{T}^{\leq0},\mathcal{T}^{\geq0})$ on $\mathcal{T}$ and a torsion class $\mathcal{X}$ of the heart $\mathcal{T}^{0}$, we define the following two full subcategories of $\mathcal{T}$:
\begin{center}
$\mu^{\mathrm{L}}_{\mathcal{X}}\mathcal{T}^{\leq0}:=\mathcal{T}^{\leq-1}*\mathcal{X},~\mu^{\mathrm{L}}_{\mathcal{X}}\mathcal{T}^{\geq0}:=\Sigma\mathcal{X}^{\perp_{\mathcal{T}^{0}}}*\mathcal{T}^{\geq0}$.
\end{center}
By \cite[Chapter 1, Proposition 2.1]{HRS}, we have the following lemma:

\begin{lem} $(1)$ The pair $(\mu^{\mathrm{L}}_{\mathcal{X}}\mathcal{T}^{\leq0},\mu^{\mathrm{L}}_{\mathcal{X}}\mathcal{T}^{\geq0})$ is also a  $t$-structure on $\mathcal{T}$, and its heart is $\Sigma\mathcal{X}^{\perp_{\mathcal{T}^{0}}}*\mathcal{X}$.

$(2)$ If $(\mathcal{T}^{\leq0},\mathcal{T}^{\geq0})$ is bounded, then the pair $(\mu^{\mathrm{L}}_{\mathcal{X}}\mathcal{T}^{\leq0},\mu^{\mathrm{L}}_{\mathcal{X}}\mathcal{T}^{\geq0})$ is also a bounded $t$-structure on $\mathcal{T}$.
\end{lem}

The $t$-structure $(\mu^{\mathrm{L}}_{\mathcal{X}}\mathcal{T}^{\leq0},\mu^{\mathrm{L}}_{\mathcal{X}}\mathcal{T}^{\geq0})$ is called \emph{HRS-tilt} of $(\mathcal{T}^{\leq0},\mathcal{T}^{\geq0})$ with respect to $\mathcal{X}$. It is easy to see that
\begin{center}
$\mathcal{T}^{\leq-1}\subseteq\mu^{\mathrm{L}}_{\mathcal{X}}\mathcal{T}^{\leq0}\subseteq\mathcal{T}^{\leq0}$ and $\mathcal{T}^{\geq1}\subseteq\mu^{\mathrm{L}}_{\mathcal{X}}\mathcal{T}^{\geq0}\subseteq\mathcal{T}^{\geq0}$.
\end{center}

\begin{lem}\label{lem2} Let $\mathfrak{B}$ be a good metric and $(\mathcal{T}^{\leq0},\mathcal{T}^{\geq0})$ a $t$-structure on $\mathcal{T}$. If $(\mathcal{T}^{\leq0},\mathcal{T}^{\geq0})$ is extendable with respect to $\mathfrak{B}$, then so is $(\mu^{\mathrm{L}}_{\mathcal{X}}\mathcal{T}^{\leq0},\mu^{\mathrm{L}}_{\mathcal{X}}\mathcal{T}^{\geq0})$ with respect to $\mathfrak{B}$.
\end{lem}
\begin{proof} By assumption, $(\mathcal{T}^{\leq0},\mathcal{T}^{\geq0})$ is extendable with respect to the good metric $\mathfrak{B}$, then there exists an $n\in\mathbb{N}$ such that $\mathfrak{B}_{n}\subseteq\mathcal{T}^{\leq0}$. Note that $\Sigma\mathfrak{B}_{n}\subseteq\Sigma\mathcal{T}^{\leq0}=\mathcal{T}^{\leq-1}\subseteq\mathcal{T}^{\leq-1}*\mathcal{X}$. So there exists an $m=n+1\in\mathbb{N}$ such that $\mathfrak{B}_{m}\subseteq\mathcal{T}^{\leq-1}*\mathcal{X}$. Thus $(\mu^{\mathrm{L}}_{\mathcal{X}}\mathcal{T}^{\leq0},\mu^{\mathrm{L}}_{\mathcal{X}}\mathcal{T}^{\geq0})$ is extendable respect to the good metric $\mathfrak{B}$.
\end{proof}

\bigskip
\section{\bf Main results}
In this section, we give the proof of the Theorem \ref{thmA}, and apply the result to the iteration HRS-tilt.

Throughout this section, let $\mathcal{T}$ be a $K$-linear essentially small triangulated category with shift functor $\Sigma$ and a good metric $\mathfrak{B}$.

\begin{lem}\label{lem0} Let $\mathcal{X}$ be a full subcategory of $\mathcal{T}$. Then
\begin{center}
$\mathfrak{G}(\mathcal{X}^{\perp_{\mathcal{T}}})\subseteq\mathfrak{G}(\mathcal{X})^{\perp_{\mathfrak{G}(\mathcal{T})}}$.
\end{center}
\end{lem}
\begin{proof} Let $E\in\mathfrak{G}(\mathcal{X}^{\perp_{\mathcal{T}}})$ and $F\in\mathfrak{G}(\mathcal{X})$. Then there exist Cauchy sequences $\{E_{\bullet}\}$ and $\{F_{\bullet}\}$ with respect to the good metric $\mathfrak{B}$ in $\mathcal{T}$ with $E_{k}\in\mathcal{X}^{\perp_{\mathcal{T}}}$ and $F_{m}\in\mathcal{X}$ for any $k,~m\in\mathbb{N}$ such that $E\simeq\displaystyle\operatorname*{colim}_{k\to}Y(E_{k})$ and $F\simeq\displaystyle\operatorname*{colim}_{m\to}Y(F_{m})$. Note that for any $k,~m\in\mathbb{N}$, $\mathrm{Hom}_{\mathcal{T}}(F_{m},E_{k})=0$. So
\begin{align*}
{\rm Hom}_{\mathcal{T}\text{-Mod}}(F, E)
& = {\rm Hom}_{\mathcal{T}\text{-Mod}}(\displaystyle\operatorname*{colim}_{m\to}Y(F_{m}),\displaystyle\operatorname*{colim}_{k\to}Y(E_{k})) \\
&\cong \varprojlim\nolimits_{m }\displaystyle\operatorname*{colim}_{k\to}{\rm Hom}_{\mathcal{T}}(F_{m}, E_{k})\\
&= 0.
\end{align*}Thus $E\in F^{\perp}$. This implies that $\mathfrak{G}(\mathcal{X}^{\perp_{\mathcal{T}}})\subseteq\mathfrak{G}(\mathcal{X})^{\perp_{\mathfrak{G}(\mathcal{T})}}$.
\end{proof}

We are particularly interested in the following proposition.

\begin{prop} Let $(\mathcal{T}^{\leq0},\mathcal{T}^{\geq0})$ be a $t$-structure on $\mathcal{T}$ and $\mathcal{X}$ a torsion class of the heart $\mathcal{T}^{0}$. Then
\begin{center}
$\mathfrak{G}(\mathcal{X}^{\perp_{\mathcal{T}^{0}}})=\mathfrak{G}(\mathcal{X})^{\perp_{Y(\mathcal{T}^{0})}}$.
\end{center}
\end{prop}
\begin{proof} By Lemma \ref{lem0}, the inclusion relation $\mathfrak{G}(\mathcal{X}^{\perp_{\mathcal{T}^{0}}})\subseteq\mathfrak{G}(\mathcal{X})^{\perp_{Y(\mathcal{T}^{0})}}$ holds. Let $E\in\mathfrak{G}(\mathcal{X})^{\perp_{Y(\mathcal{T}^{0})}}$ and $F\in\mathfrak{G}(\mathcal{X})$. Note that $F\simeq\displaystyle\operatorname*{colim}_{m\to}Y(F_{m})$, where $\{F_{\bullet}\}$ is a Cauchy sequence in $\mathcal{T}$ with respect to the good metric $\mathfrak{B}$ with $F_{m}\in\mathcal{X}$ for $m\in\mathbb{N}$. As $E\in Y(\mathcal{T}^{0})$, there exists an object $T\in\mathcal{T}^{0}$ such that $E=Y(T)$. Then
\begin{align*}
\varprojlim\nolimits_{m }{\rm Hom}_{\mathcal{T}}(F_{m}, T)
& \cong {\rm Hom}_{\mathcal{T}\text{-Mod}}(\displaystyle\operatorname*{colim}_{m\to}Y(F_{m}),Y(T)) \\
&= {\rm Hom}_{\mathcal{T}\text{-Mod}}(F, E)\\
&= 0.
\end{align*}It implies that $T\in\mathcal{X}^{\perp}$. So $E\in Y(\mathcal{X}^{\perp})\simeq\mathfrak{G}(\mathcal{X}^{\perp})$ by Lemma \ref{thm2} (2), and hence $\mathfrak{G}(\mathcal{X})^{\perp_{Y(\mathcal{T}^{0})}}\subseteq\mathfrak{G}(\mathcal{X}^{\perp_{\mathcal{T}^{0}}})$.
\end{proof}

In the following, we show that the lifting of a torsion class is also torsion class in the completion of a triangulated category.

\begin{lem}\label{lem1}  Let $(\mathcal{T}^{\leq0},\mathcal{T}^{\geq0})$ be a $t$-structure on $\mathcal{T}$ and $\mathcal{X}$ a torsion class of the heart $\mathcal{T}^{0}$. If $(\mathcal{T}^{\leq0},\mathcal{T}^{\geq0})$ is extendable with respect to $\mathfrak{B}$, then $\mathfrak{G}(\mathcal{X})$ is a torsion class of the heart $Y(\mathcal{T}^{0})$.
\end{lem}
\begin{proof} It follows from Lemma \ref{thm2} (1) that $(\mathfrak{G}(\mathcal{T}^{\leq0}),\mathfrak{G}(\mathcal{T}^{\geq0}))$ is a $t$-structure on $\mathfrak{G}(\mathcal{T})$ with the heart $Y(\mathcal{T}^{0})$. Since $\mathcal{X}\subseteq\mathcal{T}^{0}\subseteq\mathcal{T}^{\geq0}$, we have $\mathfrak{G}(\mathcal{X})\simeq Y(\mathcal{X})\subseteq Y(\mathcal{T}^{0})$ by Lemma \ref{thm2} (2), and hence $\mathfrak{G}(\mathcal{X})$ is a full subcategory of $Y(\mathcal{T}^{0})$. For any $E\in Y(\mathcal{T}^{0})$, there exists $t\in\mathcal{T}^{0}$ such that $E=Y(t)$. Note that $\mathcal{X}$ is a torsion class of the heart $\mathcal{T}^{0}$, there is an exact sequence
\begin{center}
$0\rightarrow x\rightarrow t\rightarrow y\rightarrow0$
\end{center}
 with $x\in\mathcal{X}$ and $y\in\mathcal{X}^{\perp}$. As $Y$ is exact, it follows that the sequence $0\rightarrow Y(x)\rightarrow Y(t)\rightarrow Y(y)\rightarrow0$ is exact. Obviously, $Y(x)\in Y(\mathcal{X})\simeq\mathfrak{G}(\mathcal{X})$. Since $\mathcal{X}^{\perp}\subseteq\mathcal{T}^{0}\subseteq\mathcal{T}^{\geq0}$, we have $Y(y)\in Y(\mathcal{X}^{\perp}) \simeq\mathfrak{G}(\mathcal{X}^{\perp})\subseteq(\mathfrak{G}(\mathcal{X}))^{\perp}$ by Lemma \ref{thm2} (2) and Lemma \ref{lem0}. Thus $\mathfrak{G}(\mathcal{X})$ is a torsion class of the heart $Y(\mathcal{T}^{0})$.
\end{proof}

\begin{lem}\label{lem3}  Let $(\mathcal{T}^{\leq0},\mathcal{T}^{\geq0})$ be a $t$-structure on $\mathcal{T}$ and $\mathcal{X}$ a torsion class of the heart $\mathcal{T}^{0}$. If $(\mathcal{T}^{\leq0},\mathcal{T}^{\geq0})$ is extendable with respect $\mathfrak{B}$, then
\begin{center}
$(\mu^{\mathrm{L}}_{\mathfrak{G}(\mathcal{X})}\mathfrak{G}(\mathcal{T}^{\leq0}),\mu^{\mathrm{L}}_{\mathfrak{G}(\mathcal{X})}\mathfrak{G}(\mathcal{T}^{\geq0}))$ \end{center}is a $t$-structure on $\mathfrak{G}(\mathcal{T}),$
where
\begin{center}
$\mu^{\mathrm{L}}_{\mathfrak{G}(\mathcal{X})}\mathfrak{G}(\mathcal{T}^{\leq0}):=\mathfrak{G}(\mathcal{T}^{\leq-1})*\mathfrak{G}(\mathcal{X})$ and $\mu^{\mathrm{L}}_{\mathfrak{G}(\mathcal{X})}\mathfrak{G}(\mathcal{T}^{\geq0}):=\Sigma(\mathfrak{G}(\mathcal{X})^{\perp_{Y(\mathcal{T}^{0})}})*\mathfrak{G}(\mathcal{T}^{\geq0})$.
\end{center}
\end{lem}
\begin{proof} Since $(\mathfrak{G}(\mathcal{T}^{\leq0}),\mathfrak{G}(\mathcal{T}^{\geq0}))$ is a $t$-structure on $\mathfrak{G}(\mathcal{T})$ by Lemma \ref{thm2} (1) and $\mathfrak{G}(\mathcal{X})$ is a torsion class of the heart $Y(\mathcal{T}^{0})$ by Lemma \ref{lem1}, we have $(\mu^{\mathrm{L}}_{\mathfrak{G}(\mathcal{X})}\mathfrak{G}(\mathcal{T}^{\geq0}),\mu^{\mathrm{L}}_{\mathfrak{G}(\mathcal{X})}\mathfrak{G}(\mathcal{T}^{\geq0}))$ is a $t$-structure on $\mathfrak{G}(\mathcal{T})$, where
$\mu^{\mathrm{L}}_{\mathfrak{G}(\mathcal{X})}\mathfrak{G}(\mathcal{T}^{\leq0})=\mathfrak{G}(\mathcal{T}^{\leq-1})*\mathfrak{G}(\mathcal{X})$ and $\mu^{\mathrm{L}}_{\mathfrak{G}(\mathcal{X})}\mathfrak{G}(\mathcal{T}^{\geq0})=\Sigma(\mathfrak{G}(\mathcal{X})^{\perp_{Y(\mathcal{T}^{0})}})*\mathfrak{G}(\mathcal{T}^{\geq0})$.
\end{proof}

Now we are the position to give the proof of Theorem \ref{thmA}.

\begin{proof}Let $E\in\mathfrak{G}(\mu^{\mathrm{L}}_{\mathcal{X}}\mathcal{T}^{\leq0})$. Then $E\simeq\displaystyle\operatorname*{colim}_{n\to}Y(E_{n})$, where $\{E_{\bullet},e_{\bullet}\}$ is a Cauchy sequence in $\mathcal{T}$ with respect to $\mathfrak{B}$ with $E_{n}\in\mathcal{T}^{\leq-1}*\mathcal{X}$ for $n\in\mathbb{N}$. One has the following diagram
\begin{equation}\begin{split}\label{diag*}
\xymatrix@C=25pt@R=20pt{
F_{n}\ar[r]^{a} &E_{n}\ar[d]^{e_{n}}\ar[r]^{b}&X_{n}\ar[r]^{c} &\Sigma F_{n}&\\
F_{n+1}\ar[r]^{a'}& E_{n+1}\ar[r]^{b'}&X_{n+1}\ar[r]^{c'}&\Sigma F_{n+1}&}
\end{split}\end{equation}
with $X_{n}\in\mathcal{X}~\mathrm{and}~F_{n}\in\mathcal{T}^{\leq-1} ~\mathrm{for }~ n\in\mathbb{N}$. Applying functor $\mathrm{Hom}_{\mathcal{T}}(F_{n},-)$ to the second row of the diagram (3.1), we obtain a long exact sequence
\begin{center}
$\cdots\rightarrow\mathrm{Hom}_{\mathcal{T}}(F_{n},\Sigma^{-1} X_{n+1})\rightarrow\mathrm{Hom}_{\mathcal{T}}(F_{n},F_{n+1})\rightarrow\mathrm{Hom}_{\mathcal{T}}(F_{n},E_{n+1})\rightarrow\mathrm{Hom}_{\mathcal{T}}(F_{n},X_{n+1})\rightarrow\cdots$.
\end{center}
Since $F_{n}\in\mathcal{T}^{\leq-1}$, $X_{n+1}\in\mathcal{X}\subseteq\mathcal{T}^{\geq0}$ and $\mathrm{Hom}_{\mathcal{T}}(\mathcal{T}^{\leq-1},\mathcal{T}^{\geq0})=0$, we have $\mathrm{Hom}_{\mathcal{T}}(F_{n},X_{n+1})=0$. Thus there exists a morphism $f_{n}:F_{n}\rightarrow F_{n+1}$ such that $e_{n}a=a'f_{n}$. Assume that there exists a morphism $h:F_{n}\rightarrow F_{n+1}$ such that $e_{n}a=a'h$. Then $a'f_{n}=a'h$. Hence $a'(f_{n}-h)=0$. So $f_{n}-h$ is from certain morphism $F_{n}\rightarrow \Sigma^{-1}X_{n+1}$. Note that $\Sigma F_{n}\in\Sigma\mathcal{T}^{\leq-1}=\mathcal{T}^{\leq-2}\subseteq\mathcal{T}^{\leq-1}$ and  $X_{n+1}\in\mathcal{X}\subseteq\mathcal{T}^{\geq0}$. Then
\begin{center}
$\mathrm{Hom}_{\mathcal{T}}(F_{n},\Sigma^{-1}X_{n+1})\cong\mathrm{Hom}_{\mathcal{T}}(\Sigma F_{n},X_{n+1})=0.$
\end{center}
So $f_{n}=h$. Similarly, there exists a unique morphism $x_{n}:X_{n}\rightarrow X_{n+1}$ such that $b'e_{n}=x_{n}b$. Therefore we obtain the following commutative diagram
\begin{center} $\xymatrix@C=25pt@R=20pt{
F_{n}\ar[d]^{f_{n}}\ar[r]^a &E_{n}\ar[d]^{e_{n}}\ar[r]^b&X_{n}\ar[d]^{x_{n}}\ar[r]^c &\Sigma F_{n}\ar[d]^{\Sigma f_{n}}&\\
F_{n+1}\ar[r]^{a'}& E_{n+1}\ar[r]^{b'}&X_{n+1}\ar[r]^{c'}&\Sigma F_{n+1}.&}$
\end{center}Note that the Yoneda functor $Y(-)$ and the functor $\displaystyle\operatorname*{colim}_{n\to}(-)$ are exact. There is a long exact sequence:
\begin{equation}\begin{split}\label{diag*}
\displaystyle\operatorname*{colim}_{n\to}Y(F_{n})\rightarrow\displaystyle\operatorname*{colim}_{n\to}Y(E_{n})\rightarrow
\displaystyle\operatorname*{colim}_{n\to}Y(X_{n})\rightarrow\Sigma(\displaystyle\operatorname*{colim}_{n\to}Y(F_{n})).
\end{split}\end{equation}
in $\mathcal{T}\text{-Mod}$. Next, we prove that $\{F_{\bullet},f_{\bullet}\}$ and $\{X_{\bullet},x_{\bullet}\}$ are Cauchy sequences in $\mathcal{T}$.

Define $Z_{n}=\mathrm{Cone}(e_{n})$, $Z_{n}^{'}=\mathrm{Cone}(f_{n})$ and $Z_{n}^{''}=\mathrm{Cone}(x_{n})$. Then we obtain triangles $E_{n}\rightarrow E_{n+1}\rightarrow Z_{n}\rightarrow \Sigma E_{n}$, $X_{n}\rightarrow X_{n+1}\rightarrow Z_{n}^{''}\rightarrow \Sigma X_{n}$ and $F_{n}\rightarrow F_{n+1}\rightarrow Z_{n}^{'}\rightarrow \Sigma F_{n}$ in $\mathcal{T}$. By the $3\times3$ lemma of triangles, there exist dotted morphisms making the following diagram commute:
\begin{center} $\xymatrix@C=25pt@R=20pt{
\Sigma^{-1}X_{n}\ar[d]^{\Sigma^{-1}c}\ar[r]^{\Sigma^{-1}x_{n}} &\Sigma^{-1}X_{n+1}\ar[d]^{\Sigma^{-1}c'}\ar[r]^{}&\Sigma^{-1}Z_{n}^{''}\ar@{-->}[d]^{}&\\
F_{n}\ar[d]^{a}\ar[r]^{f_{n}} &F_{n+1}\ar[d]^{a'}\ar[r]^{}&Z_{n}^{'}\ar@{-->}[d]^{g'}\ar[r]^{} &\Sigma F_{n} \ar[d]^{\Sigma a}&\\
E_{n}\ar[d]^{b}\ar[r]^{e_{n}} &E_{n+1}\ar[d]^{b'}\ar[r]^{} &Z_{n}\ar@{-->}[d]^{g}\ar[r]^{} &\Sigma E_{n}&\\
X_{n}\ar[r]^{x_{n}}& X_{n+1}\ar[r]^{}&Z_{n}^{''}.&}$
\end{center}
Since $F_{n+1}\in\mathcal{T}^{\leq-1},~ \Sigma F_{n}\in\Sigma\mathcal{T}^{\leq-1}=\mathcal{T}^{\leq-2}\subseteq\mathcal{T}^{\leq-1}$ and $\mathcal{T}^{\leq-1}$ is closed under extensions, we have $Z_{n}^{'}\in\mathcal{T}^{\leq-1}$. Similarly, $Z_{n}^{''}\in\mathcal{T}^{\geq-1}$ as $\mathcal{T}^{\geq-1}$ is closed under extensions. By assumption, $(\mathcal{T}^{\leq0},\mathcal{T}^{\geq0})$ is extendable with respect to the good metric $\mathfrak{B}$, one has a natural number $k$ such that $\mathfrak{B}_{k}\subseteq\mathcal{T}^{\leq0}$. Hence $\mathfrak{B}_{k+1}\subseteq\mathcal{T}^{\leq-1}*\mathcal{X}$. Since $\{E_{\bullet},e_{\bullet}\}$ is a Cauchy sequence, for each $l\in\mathbb{N}$, there exists an $N_{l}\in\mathbb{N}$ such that $Z_{n}\in\mathfrak{B}_{l}$ for all $n\geq N_{l}$. Consider $l\geq k+2$ and $n\geq N_{l}$. Then $Z_{n}\in\mathfrak{B}_{l}\subseteq\mathfrak{B}_{k+2}\subseteq\Sigma^{2}\mathfrak{B}_{k}\subseteq\Sigma^{2}\mathcal{T}^{\leq0}=\mathcal{T}^{\leq-2}$, where the middle inclusion relation is from Remark \ref{rmk0}. As $Z_{n}^{''}\in\mathcal{T}^{\geq-1}$, it follows that $\mathrm{Hom}_{\mathcal{T}}(Z_{n},Z_{n}^{''})=0$, and so $g=0$. Thus $Z_{n}^{'}\cong \Sigma^{-1}Z_{n}^{''}\oplus Z_{n}$. Note that $\Sigma^{-1}Z_{n}^{''}\in\Sigma^{-1}\mathcal{T}^{\geq-1}=\mathcal{T}^{\geq0}$, $Z_{n}^{'}\in\mathcal{T}^{\leq-1}$ and $\mathrm{Hom}_{\mathcal{T}}(\mathcal{T}^{\leq-1},\mathcal{T}^{\geq0})=0$. Then
\begin{center}
$\mathrm{Hom}_{\mathcal{T}}(Z_{n}^{'},\Sigma^{-1}Z_{n}^{''})\cong\mathrm{Hom}_{\mathcal{T}}(\Sigma^{-1}Z_{n}^{''}\oplus Z_{n},\Sigma^{-1}Z_{n}^{''})=0$.
\end{center}It yields that $Z_{n}^{''}=0$, and so $x_{n}$ and $g'$ are isomorphisms. Since $\{E_{\bullet},e_{\bullet}\}$ is a Cauchy sequence and $g'$ is an isomorphism, we have $\{F_{\bullet},f_{\bullet}\}$ is a Cauchy sequence in $\mathcal{T}$. The fact that $x_{n}$ is an isomorphism implies that $\{X_{\bullet},x_{\bullet}\}$ is a Cauchy sequence in $\mathcal{T}$. Thus
\begin{center}
$\displaystyle\operatorname*{colim}_{n\to}Y(F_{n})\in\mathfrak{L}(\mathcal{T}^{\leq-1})$ and $\displaystyle\operatorname*{colim}_{n\to}Y(X_{n})\in\mathfrak{L}(\mathcal{X})$.
\end{center}Put $p=N_{k+2}$. Since $Z_{n}^{''}=0$ for each $n\geq p$, the sequence $\{X_{\bullet},x_{\bullet}\}$ is stable and
\begin{center}
$\displaystyle\operatorname*{colim}_{n\to}Y(X_{n})\simeq Y(X_{p})\in Y(\mathcal{X})$.
\end{center}

Note that $\mathfrak{B}_{k+2}\subseteq\Sigma^{2}\mathfrak{B}_{k}\subseteq\mathcal{T}^{\leq-2}\subseteq\mathcal{T}^{\leq-1}$ and $(\mathcal{T}^{\leq-1})^{\perp}=\mathcal{T}^{\geq0}$. Then $\mathcal{T}^{\geq0}\subseteq(\mathfrak{B}_{k+2})^{\perp}$. As $Y(\mathfrak{B}_{k+2}^{\perp})\in\mathfrak{C}(\mathcal{T})$ by the definition of $\mathfrak{C}(\mathcal{T})$, it yields that $Y(\mathcal{X})\subseteq Y(\mathcal{T}^{\geq0})\subseteq\mathfrak{C}(\mathcal{T})$, and so $\displaystyle\operatorname*{colim}_{n\to}Y(X_{n})\in Y(\mathcal{X})\in \mathfrak{C}(\mathcal{T})$. Thus $\displaystyle\operatorname*{colim}_{n\to}Y(X_{n})\in \mathfrak{G}(\mathcal{X})$. Let $T\in\mathfrak{B}_{k+2}$. Applying $\mathrm{Hom}(Y(T),-)$ to the diagram (3.2), we obtain the following commutative diagram of exact sequences:\begin{center} $\xymatrix@C=5pt@R=15pt{
(Y(T),\Sigma^{-1}(\displaystyle\operatorname*{colim}_{n\to}Y(X_{n})))\ar[d]^{\cong}\ar[r]^{}&(Y(T),\displaystyle\operatorname*{colim}_{n\to}Y(F_{n}))\ar[d]^{\cong}\ar[r]^{} &(Y(T),\displaystyle\operatorname*{colim}_{n\to}Y(E_{n}))\ar[d]^{\cong}\ar[r]^{}&(Y(T),\displaystyle\operatorname*{colim}_{n\to}Y(X_{n}))\ar[d]^{\cong} &\\
(T,\Sigma^{-1}X_{p})\ar[r]^{}& \displaystyle\operatorname*{colim}_{n\to}(T,F_{n})\ar[r]^{}&\displaystyle\operatorname*{colim}_{n\to}(T,E_{n})\ar[r]^{}&(T,X_{p}).&}$
\end{center}
Since $\mathcal{T}^{\geq0}\subseteq(\mathfrak{B}_{k+2})^{\perp}$, $X_{p}\in\mathcal{X}\subseteq\mathcal{T}^{\geq0}$ and $\Sigma^{-1}X_{p}\in\Sigma^{-1}\mathcal{T}^{\geq0}\subseteq\mathcal{T}^{\geq0}$, we have  $\mathrm{Hom}_{\mathcal{T}}(T,\Sigma^{-1}X_{p})=0=\mathrm{Hom}_{\mathcal{T}}(T,X_{p})$. So $\displaystyle\operatorname*{colim}_{n\to}\mathrm{Hom}_{\mathcal{T}}(T,F_{n})\rightarrow\displaystyle\operatorname*{colim}_{n\to}\mathrm{Hom}_{\mathcal{T}}(T,E_{n})$ is an isomorphism. Hence $\displaystyle\operatorname*{colim}_{n\to}Y(E_{n})\in\mathfrak{C}(\mathcal{T})$ if and only if $\displaystyle\operatorname*{colim}_{n\to}Y(F_{n})\in\mathfrak{C}(\mathcal{T})$. It yields that $\displaystyle\operatorname*{colim}_{n\to}Y(F_{n})\in \mathfrak{C}(\mathcal{T})$. So $\displaystyle\operatorname*{colim}_{n\to}Y(F_{n})\in \mathfrak{G}(\mathcal{T}^{\leq-1})$. We finish the proof of the inclusion relation $\mathfrak{G}(\mu^{\mathrm{L}}_{\mathcal{X}}\mathcal{T}^{\leq0})\subseteq\mu^{\mathrm{L}}_{\mathfrak{G}(\mathcal{X})}\mathfrak{G}(\mathcal{T}^{\leq0})$.

For the reverse containment, since $\mathcal{T}^{\leq-1}\subseteq\mathcal{T}^{\leq-1}\ast\mathcal{X}$ and $\mathcal{X}\subseteq\mathcal{T}^{\leq-1}\ast\mathcal{X}$, we have
\begin{center}
$\mathfrak{G}(\mathcal{T}^{\leq-1})\subseteq\mathfrak{G}(\mathcal{T}^{\leq-1}\ast\mathcal{X})$ and $\mathfrak{G}(\mathcal{X})\subseteq\mathfrak{G}(\mathcal{T}^{\leq-1}\ast\mathcal{X})$.
\end{center}By Lemma \ref{lem2}, $(\mu^{\mathrm{L}}_{\mathcal{X}}\mathcal{T}^{\leq0},\mu^{\mathrm{L}}_{\mathcal{X}}\mathcal{T}^{\geq0})$ is an extendable $t$-structure with respect to the good metric $\mathfrak{B}$. It follows from Lemma \ref{thm2} (1) that $(\mathfrak{G}(\mu^{\mathrm{L}}_{\mathcal{X}}\mathcal{T}^{\leq0}),\mathfrak{G}(\mu^{\mathrm{L}}_{\mathcal{X}}\mathcal{T}^{\geq0}))$ is a $t$-structure on $\mathfrak{G}(\mathcal{T})$. Then $\mathfrak{G}(\mu^{\mathrm{L}}_{\mathcal{X}}\mathcal{T}^{\leq0})=\mathfrak{G}(\mathcal{T}^{\leq-1}\ast\mathcal{X})$ is closed under extensions as aisle of $(\mathfrak{G}(\mu^{\mathrm{L}}_{\mathcal{X}}\mathcal{T}^{\leq0}),\mathfrak{G}(\mu^{\mathrm{L}}_{\mathcal{X}}\mathcal{T}^{\geq0}))$. So $\mathfrak{G}(\mathcal{T}^{\leq-1})\ast\mathfrak{G}(\mathcal{X})\subseteq\mathfrak{G}(\mathcal{T}^{\leq-1}\ast\mathcal{X})$.

Since $(\mu^{\mathrm{L}}_{\mathcal{X}}\mathcal{T}^{\leq0},\mu^{\mathrm{L}}_{\mathcal{X}}\mathcal{T}^{\geq0})$ is an extendable $t$-structure with respect to the good metric $\mathfrak{B}$ by Lemma \ref{lem2}, we have $(\mathfrak{G}(\mu^{\mathrm{L}}_{\mathcal{X}}\mathcal{T}^{\leq0}),\mathfrak{G}(\mu^{\mathrm{L}}_{\mathcal{X}}\mathcal{T}^{\geq0}))$ is a $t$-structure on $\mathfrak{G}(\mathcal{T})$. It follows from Lemma \ref{lem3} that $(\mu^{\mathrm{L}}_{\mathfrak{G}(\mathcal{X})}\mathfrak{G}(\mathcal{T}^{\leq0}),\mu^{\mathrm{L}}_{\mathfrak{G}(\mathcal{X})}\mathfrak{G}(\mathcal{T}^{\geq0}))$ is a $t$-structure on $\mathfrak{G}(\mathcal{T})$. Thus \begin{align*}
\mathfrak{G}(\mu^{\mathrm{L}}_{\mathcal{X}}\mathcal{T}^{\geq0})
& = (\mathfrak{G}(\mu^{\mathrm{L}}_{\mathcal{X}}\mathcal{T}^{\leq-1}))^{\perp} \\
& = (\mu^{\mathrm{L}}_{\mathfrak{G}(\mathcal{X})}\mathfrak{G}(\mathcal{T}^{\leq-1}))^{\perp}\\
& = \mu^{\mathrm{L}}_{\mathfrak{G}(\mathcal{X})}\mathfrak{G}(\mathcal{T}^{\geq0}).
\end{align*}
\end{proof}

In the subsequent content of this paper, for a $t$-structure $(\mathcal{T}^{\leq0},\mathcal{T}^{\geq0})$ on $\mathcal{T}$ and a torsion class $\mathcal{X}$ of the heart $\mathcal{T}^{0}$, we use $T_{\mathrm{HRS}}^{\mathcal{X}}$ to represent the HRS-tilt process of $(\mathcal{T}^{\leq0},\mathcal{T}^{\geq0})$.

\begin{cor}\label{cor0}  Let $(\mathcal{T}^{\leq0},\mathcal{T}^{\geq0})$ be a $t$-structure on $\mathcal{T}$ and $\mathcal{X}$ a torsion class of the heart $\mathcal{T}^{0}$. If $(\mathcal{T}^{\leq0},\mathcal{T}^{\geq0})$ is extendable with respect to $\mathfrak{B}$, then we have the following commutative diagram of abelian categories:
\begin{center} $\xymatrix@C=35pt@R=20pt{
\mathcal{T}^{0}\ar[d]^{\simeq}\ar[r]^{T_{\mathrm{HRS}}^{\mathcal{X}}} &\Sigma\mathcal{X}^{\perp_{\mathcal{T}^{0}}}*\mathcal{X}\ar[d]^{\simeq}&\\
\mathfrak{G}(\mathcal{T}^{0})\ar[r]^{T_{\mathrm{HRS}}^{\mathfrak{G}(\mathcal{X})}}& \mathfrak{G}(\Sigma\mathcal{X}^{\perp_{\mathcal{T}^{0}}}*\mathcal{X}).&}$
\end{center}
\end{cor}
\begin{proof} By Theorem \ref{thmA}, we have the commutative diagram
\begin{center} $\xymatrix@C=25pt@R=20pt{
(\mathcal{T}^{\leq0},\mathcal{T}^{\geq0})\ar[d]^{\mathfrak{G}(-)}\ar[r]^{T_{\mathrm{HRS}}^{\mathcal{X}}} &(\mu^{\mathrm{L}}_{\mathcal{X}}\mathcal{T}^{\leq0},\mu^{\mathrm{L}}_{\mathcal{X}}\mathcal{T}^{\geq0})\ar[d]^{\mathfrak{G}(-)}&\\
(\mathfrak{G}(\mathcal{T}^{\leq0}),\mathfrak{G}(\mathcal{T}^{\geq0}))\ar[r]^{T_{\mathrm{HRS}}^{\mathfrak{G}(\mathcal{X})}}& (\mathfrak{G}(\mu^{\mathrm{L}}_{\mathcal{X}}\mathcal{T}^{\leq0}),\mathfrak{G}(\mu^{\mathrm{L}}_{\mathcal{X}}\mathcal{T}^{\geq0})),&}$
\end{center}and hence $\mathfrak{G}(\Sigma\mathcal{X}^{\perp_{\mathcal{T}^{0}}}*\mathcal{X})=\Sigma(\mathfrak{G}(\mathcal{X})^{\perp_{\mathfrak{G}(\mathcal{T}^{0})}})\ast\mathfrak{G}(\mathcal{X})$. Note that $(\mu^{\mathrm{L}}_{\mathcal{X}}\mathcal{T}^{\leq0},\mu^{\mathrm{L}}_{\mathcal{X}}\mathcal{T}^{\geq0})$ is a $t$-structure on $\mathcal{T}$ with the heart $\Sigma\mathcal{X}^{\perp_{\mathcal{T}^{0}}}*\mathcal{X}$ and it is extendable with respect to the good metric $\mathfrak{B}$ by Lemma \ref{lem2}. It follows from Lemma \ref{thm2} (2) that
\begin{center}
$\mathcal{T}^{0}\rightarrow\mathfrak{G}(\mathcal{T}^{0})$ and $\Sigma\mathcal{X}^{\perp_{\mathcal{T}^{0}}}*\mathcal{X}\rightarrow\mathfrak{G}(\Sigma\mathcal{X}^{\perp_{\mathcal{T}^{0}}}*\mathcal{X})$
\end{center}
 are equivalences. So the conclusion holds.
\end{proof}

We recall a fact of iteration HRS-tilt from \cite[Chapter 1]{HRS}. Let $(\mathcal{T}_{1}^{\leq0},\mathcal{T}_{1}^{\geq0})$ be a $t$-structure on $\mathcal{T}$ with a torsion pair $\tau_{1}=(\mathcal{X},\mathcal{Y})$ of the heart $\mathcal{T}_{1}^{0}$, we have
\begin{align*}
&\mathcal{T}_{2}^{0}=\Sigma\mathcal{Y}\ast\mathcal{X} &&  \tau_{2}=(\Sigma\mathcal{Y},\mathcal{X}),\\
&\mathcal{T}_{3}^{0}=\Sigma\mathcal{X}\ast\Sigma\mathcal{Y}=\Sigma\mathcal{T}_{1}^{0} && \tau_{3}=(\Sigma\mathcal{X},\Sigma\mathcal{Y}),\\
&\mathcal{T}_{4}^{0}=\Sigma^{2}\mathcal{Y}\ast\Sigma\mathcal{X}&& \tau_{4}=(\Sigma^{2}\mathcal{Y},\Sigma\mathcal{X}),\\
&\qquad{\vdots}&& \qquad{\vdots}\\
&\mathcal{T}_{2i+1}^{0}=\Sigma^{i}\mathcal{X}\ast\Sigma^{i}\mathcal{Y}=\Sigma^{i}\mathcal{T}_{1}^{0} && \tau_{2i+1}=(\Sigma^{i}\mathcal{X},\Sigma^{i}\mathcal{Y}),\\
&\mathcal{T}_{2i+2}^{0}=\Sigma^{i+1}\mathcal{Y}\ast\Sigma^{i}\mathcal{X} && \tau_{2i+2}=(\Sigma^{i+1}\mathcal{Y},\Sigma^{i}\mathcal{X})\\
\end{align*}
for $i\geq2$.

Consider the torsion pair $(\mathcal{X},\mathcal{X}^{\perp})$ of the heart $\mathcal{T}_{1}^{0}$, we have the following corollary.

\begin{cor} Let $(\mathcal{T}_{1}^{\leq0},\mathcal{T}_{1}^{\geq0})$ be a $t$-structure on $\mathcal{T}$. Assume that $\tau_{1}=(\mathcal{X},\mathcal{X}^{\perp})$ is a torsion pair of the heart $\mathcal{T}_{1}^{0}$ and $\mathcal{T}_{i+1}^{\leq0}=\mu^{\mathrm{L}}_{\tau_{i}}\mathcal{T}_{i}^{\leq0}$, $\mathcal{T}_{i+1}^{\geq0}=\mu^{\mathrm{L}}_{\tau_{i}}\mathcal{T}_{i}^{\geq0}$ for any $i\geq1$. If $(\mathcal{T}_{1}^{\leq0},\mathcal{T}_{1}^{\geq0})$ is extendable with respect to $\mathfrak{B}$, then the squares in the following diagram are commutative
\begin{center} $\xymatrix@C=15pt@R=15pt{
(\mathcal{T}_{1}^{\leq0},\mathcal{T}_{1}^{\geq0})\ar[d]^{\mathfrak{G}(-)}\ar[r]^{T_{\mathrm{HRS}}^{\mathcal{X}}} &
(\mathcal{T}_{2}^{\leq0},\mathcal{T}_{2}^{\geq0})\ar[d]^{\mathfrak{G}(-)}\ar[r]^{T_{\mathrm{HRS}}^{\Sigma\mathcal{X}^{\perp}}}& (\mathcal{T}_{3}^{\leq0},\mathcal{T}_{3}^{\geq0})\ar[d]^{\mathfrak{G}(-)}\ar[r]^{T_{\mathrm{HRS}}^{\Sigma\mathcal{X}}}& (\mathcal{T}_{4}^{\leq0},\mathcal{T}_{4}^{\geq0})\ar[d]^{\mathfrak{G}(-)}\ar[r]^{T_{\mathrm{HRS}}^{\Sigma^{2}\mathcal{X}^{\perp}}}& \cdots \\
(\mathfrak{G}(\mathcal{T}_{1}^{\leq0}),\mathfrak{G}(\mathcal{T}_{1}^{\geq0}))\ar[r]^{T_{\mathrm{HRS}}^{\mathfrak{G}(\mathcal{X})}}& (\mathfrak{G}(\mathcal{T}_{2}^{\leq0}),\mathfrak{G}(\mathcal{T}_{2}^{\geq0}))\ar[r]^{T_{\mathrm{HRS}}^{\mathfrak{G}(\Sigma\mathcal{X}^{\perp})}}& (\mathfrak{G}(\mathcal{T}_{3}^{\leq0}),\mathfrak{G}(\mathcal{T}_{3}^{\geq0}))\ar[r]^{T_{\mathrm{HRS}}^{\mathfrak{G}(\Sigma\mathcal{X})}}& (\mathfrak{G}(\mathcal{T}_{4}^{\leq0}),\mathfrak{G}(\mathcal{T}_{4}^{\geq0}))\ar[r]^{}& \cdots .&}$
\end{center}
Further, we have \begin{center}
$\mathfrak{G}(\mathcal{T}_{i+1}^{\leq0})=\mu^{\mathrm{L}}_{\mathfrak{G}(\tau_{i})}\mathfrak{G}(\mathcal{T}_{i}^{\leq0})$,
$\mathfrak{G}(\mathcal{T}_{i+1}^{\geq0})=\mu^{\mathrm{L}}_{\mathfrak{G}(\tau_{i})}\mathfrak{G}(\mathcal{T}_{i}^{\geq0})$.
\end{center}
\end{cor}
\begin{proof} This follows from Theorem \ref{thmA}.
\end{proof}

\bigskip
\section{\bf Applications}
Let $\mathcal{T}$ be a $K$-linear essentially small triangulated category with shift functor $\Sigma$, in this section, we apply Theorem \ref{thmA} to bounded $t$-structures and ST-pairs.

Before presenting a classical example of good metric, we first review some notations and notions from \cite[Reminders 1.1 and 0.1; Definition 1.3]{N3} (see also \cite[Subsection 2.2]{BV}).

Let $\mathcal{X}$ be a class of objects in $\mathcal{T}$ and $G$ an object of $\mathcal{T}$. We use $\mathrm{smd}(\mathcal{X})$ to denote the smallest full subcategory of $\mathcal{T}$ which contains $\mathcal{X}$ and is closed under direct summands. We denote by $\mathrm{add}(\mathcal{X})$ the smallest full subcategory of $\mathcal{T}$ which contains $\mathcal{X}$ and is closed under finite direct sums and direct summands. Let $i>0$, the subcategory $\mathrm{coprod}_{i}(\mathcal{X})$ of $\mathcal{T}$ is defined inductively as follows
\begin{center}
$\mathrm{coprod}_{1}(\mathcal{X}):=\mathrm{add}(\mathcal{X}),~\mathrm{coprod}_{i+1}(\mathcal{X}):=
\mathrm{coprod}_{1}(\mathcal{X})\ast\mathrm{coprod}_{i}(\mathcal{X})$.
\end{center}
For $a,b\in\mathbb{Z}\bigcup\infty$ with $a\leq b$, define $G[a,b]:=\{\Sigma^{-j}G~|~j\in\mathbb{Z},~a\leq j\leq b\}$. For $i>0$, let
\begin{center}
$\langle G\rangle_{i}^{[a,b]}:=\mathrm{smd}(\mathrm{coprod}_{i}(G[a,b])),~\langle G\rangle^{[a,b]}:=\bigcup\limits_{i\in\mathbb{Z}}\langle G\rangle_{i}^{[a,b]}$.
\end{center}

\begin{rmk}\label{exa1} Let $G$ be an object of $\mathcal{T}$. Then the sequence $\{\langle G\rangle^{[-\infty,-n]}\}_{n\in\mathbb{N}}$ on $\mathcal{T}$ generated by $G$ forms a good metric. A good metric $\mathfrak{B}=\{\mathfrak{B}_{n}\}_{n\in\mathbb{N}}$ on $\mathcal{T}$ is called a  \emph{G-good metric} if it is equivalent to the good metric $\{\langle G\rangle^{[-\infty,-n]}\}_{n\in\mathbb{N}}$, see \cite[Definition 1.4]{BCRPZ}.
\end{rmk}

We denote the completion of $\mathcal{X}$ in Definition \ref{dfn1} with respect to a $G$-good metric on $\mathcal{T}$ by $\mathfrak{G}_{G}(\mathcal{X})$.

\begin{cor}\label{cor1} Let $G$ be an object of $\mathcal{T}$. Assume that $(\mathcal{T}^{\leq0},\mathcal{T}^{\geq0})$ is a bounded $t$-structure on $\mathcal{T}$ and $\mathcal{X}$ is a torsion class of the heart $\mathcal{T}^{0}$. Then
\begin{center} $\mathfrak{G}_{G}(\mu^{\mathrm{L}}_{\mathcal{X}}\mathcal{T}^{\leq0})=\mu^{\mathrm{L}}_{\mathfrak{G}_{G}(\mathcal{X})}\mathfrak{G}_{G}(\mathcal{T}^{\leq0})$,~
$\mathfrak{G}_{G}(\mu^{\mathrm{L}}_{\mathcal{X}}\mathcal{T}^{\geq0})=\mu^{\mathrm{L}}_{\mathfrak{G}_{G}(\mathcal{X})}\mathfrak{G}_{G}(\mathcal{T}^{\geq0})$.
\end{center}
\end{cor}
\begin{proof} By \cite[Lemma 3.3]{BCRPZ}, $(\mathcal{T}^{\leq0},\mathcal{T}^{\geq0})$ is extendable with respect to the $G$-good metric on $\mathcal{T}$. The assertions follow from Theorem \ref{thmA}.
\end{proof}

We collect the definition and some properties of ST-pairs on $\mathcal{T}$ from \cite{AMY}.

Let $\mathcal{T}$ be a Krull-Schmidt triangulated category with the shift functor $\Sigma$ and $M$ be an object of $\mathcal{T}$, $M$ is called \emph{basic} if $M\simeq M_{1}\oplus \cdots\oplus M_{n}$, where the $M_{i}$ are indecomposable and pairwise non-isomorphic. In this case, denote $|M|=n$. We note that the $n$ is unique. Recall that $M$ is called \emph{presilting} if $\mathrm{Hom}_{\mathcal{T}}(M,\Sigma^{i}M)=0$ for any $i>0$. $M$ is called \emph{silting} if $M$ is presilting and $\mathrm{thick}M=\mathcal{T}$.

For a presilting object $M$ of $\mathcal{T}$, a thick subcategory $\mathcal{S}$ of $\mathcal{T}$ and an integer $i$, we define full subcategories of $\mathcal{T}$ as follows:
\begin{center}
$\mathcal{S}_{M}^{\leq i}:=\{X\in\mathcal{S} | \mathrm{Hom}_{\mathcal{T}}(M,\Sigma^{>i}X)=0\}$,\\
$\mathcal{S}_{M}^{\geq i}:=\{X\in\mathcal{S} | \mathrm{Hom}_{\mathcal{T}}(M,\Sigma^{<i}X)=0\}$,\\
$\mathcal{S}_{M}^{ i}:=\mathcal{S}_{M}^{\leq i}\bigcap\mathcal{S}_{M}^{\geq i}$.
\end{center}
It is easy to see that the following facts hold:\\ (1) $\mathcal{S}_{M}^{\leq i}=\Sigma^{-i}\mathcal{S}_{M}^{\leq 0}=\mathcal{S}_{\Sigma^{-i}M}^{\leq 0}$ and $\mathcal{S}_{M}^{\geq i}=\Sigma^{-i}\mathcal{S}_{M}^{\geq 0}=\mathcal{S}_{\Sigma^{-i}M}^{\geq 0}$,\\
(2) $\mathcal{S}_{M}^{\leq i}=\mathcal{T}_{M}^{\leq i}\bigcap\mathcal{S}$ and $\mathcal{S}_{M}^{\geq i}=\mathcal{T}_{M}^{\geq i}\bigcap\mathcal{S}$.

\begin{df}$($\cite[Definition 4.3]{AMY}$)$\label{dfn2} Let $\mathcal{C}$ and $\mathcal{D}$ be thick subcategories of $\mathcal{T}$. The pair $(\mathcal{C}$,$\mathcal{D})$ is called an \emph{ST-pair} on $\mathcal{T}$ if there exists a silting object $M$ of $\mathcal{C}$ such that

(ST1) $\mathrm{Hom}_{\mathcal{T}}(M,T)$ is finite-dimensional for any object $T$ of $\mathcal{T}$.

(ST2) $(\mathcal{T}_{M}^{\leq0},\mathcal{T}_{M}^{\geq0})$ is a $t$-structure on $\mathcal{T}$.

(ST3) $\mathcal{T}=\bigcup_{i\in\mathbb{Z}}\mathcal{T}_{M}^{\leq i}$ and $\mathcal{D}=\bigcup_{i\in\mathbb{Z}}\mathcal{T}_{M}^{\geq i}$.
\end{df}

\begin{rmk}\label{rmk1} By \cite[Remark 4.5]{AMY}, the condition (ST3) in Definition \ref{dfn2} is equivalent to $\mathcal{T}_{M}^{\geq0}\subseteq\mathcal{D}$ and $(\mathcal{D}_{M}^{\leq0},\mathcal{D}_{M}^{\geq0})$ is a bounded $t$-structure on $\mathcal{D}$.
\end{rmk}

In the rest of this section, we assume that $\mathcal{T}$ is a $K$-linear Krull-Schmidt essentially small triangulated category with shift functor $\Sigma$. Let ($\mathcal{C}$,$\mathcal{D}$) be an ST-pair on $\mathcal{T}$, $M$ a basic silting object of $\mathcal{C}$ with $n=|M|$ and $I$ a subset of $\{1,2,\ldots,n\}$. Assume that $M_{I}=\bigoplus_{i\in I}M_{i}$ and $S_{I}=\bigoplus_{i\in I}S_{i}$, where each $M_{i}$ is indecomposable and $S_{i}$ is a complete set of pairwise non-isomorphic simple objects of $\mathcal{D}_{M}^{0}$. Define the full subcategory $\mathcal{X}_{S_{I}}$ of $\mathcal{D}_{M}^{0}$ as
\begin{center}
$\mathcal{X}_{S_{I}}:=\{X\in\mathcal{D}_{M}^{0}~|~\mathrm{Hom}_{\mathcal{T}}(X,S_{I})=0\}$.
\end{center}
Then $\mathcal{X}_{S_{I}}$ is a torsion class of $\mathcal{D}_{M}^{0}$. We denote by $\mu_{M_{I}}^{L}(M)$ the left mutation of the silting object $M$. Then $\mu_{M_{I}}^{L}(M)$ is a silting object of $\mathcal{C}$ by \cite[Theorem 2.31]{AI}. It follows from \cite[Proposition 7.7]{AMY} that $\mathcal{D}_{\mu_{M_{I}}^{L}(M)}^{\leq0}=\mathcal{\mu}_{\mathcal{X}_{S_{I}}}^{L}\mathcal{D}_{M}^{\leq0}$. Thus we have the following commutative diagram:
\begin{center} $\xymatrix@C=25pt@R=20pt{
M\ar[d]^{\Psi}\ar[r]^{\mu} &\mu_{M_{I}}^{L}(M)\ar[d]^{\Psi}&\\
(\mathcal{D}_{M}^{\leq0},\mathcal{D}_{M}^{\geq0})\ar[r]^{T_{\mathrm{HRS}}^{\mathcal{X}_{S_{I}}}}& (\mathcal{D}_{\mu_{M_{I}}^{L}(M)}^{\leq0},\mathcal{D}_{\mu_{M_{I}}^{L}(M)}^{\geq0}),&}$
\end{center}where the horizontal arrow $\mu$ indicates the left mutation of a silting object $M$ and the two vertical arrows $\Psi$ represent the correspondence in \cite[Theorem 5.5]{AMY}.

Combine the above commutative diagram with Theorem \ref{thmA}, we have the following corollary which provides the relation between the left mutation of a silting object on $\mathcal{T}$ and the HRS-tilt $t$-structure on $\mathfrak{G}_{G}(\mathcal{T})$.

\begin{cor}\label{cor2} Let $(\mathcal{C}$,$\mathcal{D})$ be an ST-pair on $\mathcal{T}$. Assume that $M\in\mathcal{C}$ is a basic silting object with $|M|=n$, $G\in\mathcal{D}$ and $I$ is a subset of $\{1,2,\ldots,n\}$. Then we have the following commutative diagram:
\begin{center} $\xymatrix@C=25pt@R=20pt{
M\ar[d]^{\mathfrak{G}_{G}(\Psi(-))}\ar[r]^{\mu} &\mu_{M_{I}}^{L}(M)\ar[d]^{\mathfrak{G}_{G}(\Psi(-))}&\\
(\mathfrak{G}_{G}(\mathcal{D}_{M}^{\leq0}),\mathfrak{G}_{G}(\mathcal{D}_{M}^{\geq0}))\ar[r]^{T_{\mathrm{HRS}}^{\mathfrak{G}(\mathcal{X}_{S_{I}})}}& (\mathfrak{G}_{G}(\mathcal{D}_{\mu_{M_{I}}^{L}(M)}^{\leq0}),\mathfrak{G}_{G}(\mathcal{D}_{\mu_{M_{I}}^{L}(M)}^{\geq0})),&}$
\end{center}
\end{cor}
\begin{proof} Since $(\mathcal{C}$,$\mathcal{D})$ is an ST-pair on $\mathcal{T}$, $(\mathcal{D}_{M}^{\leq0},\mathcal{D}_{M}^{\geq0})$ is a bounded $t$-structure on $\mathcal{D}$ by Remark \ref{rmk1}. So it follows from Corollary \ref{cor1} that the result holds.
\end{proof}

Let $A$ be a finite-dimensional $K$-algebra. We denote $\mathrm{D}^{\mathrm{b}}(\mathrm{mod}A)$ the bounded derived
category of $\mathrm{mod}A$. Then $\mathrm{D}^{\mathrm{b}}(\mathrm{mod}A)$ is a $K$-linear Krull-Schmidt triangulated category by \cite[Theorem 1]{A1} and \cite[Lemma 2.1]{HKP}. It is noticed that $\mathrm{D}^{\mathrm{b}}(\mathrm{mod}A)$ is also essentially small. We denote by $\mathrm{proj}A$ the category of finite-dimensional projective $A$-modules and $\mathrm{K}^{\mathrm{b}}(\mathrm{proj}A)$ the bounded homotopy category of $\mathrm{proj}A$.

Let
\begin{center}
$\mathrm{D}^{\leq0}_{A}=\{X\in\mathrm{D^{b}}(\mathrm{mod}A)~|~H^{i}(X)=0,~\forall~i>0\}$,\\
$\mathrm{D}^{\geq0}_{A}=\{X\in\mathrm{D^{b}}(\mathrm{mod}A)~|~H^{i}(X)=0,~\forall~i<0\}$.
\end{center}Then $(\mathrm{D}^{\leq0}_{A},\mathrm{D}^{\geq0}_{A})$ is a bounded $t$-structure on $\mathrm{D}^{\mathrm{b}}(\mathrm{mod}A)$ by \cite[Example 1.3.2(i)]{BBD}. It follows from \cite[Lemma 4.10]{AMY} that $(\mathrm{K}^{\mathrm{b}}(\mathrm{proj}A),\mathrm{D}^{\mathrm{b}}(\mathrm{mod}A))$ is an ST-pair on $\mathrm{D}^{\mathrm{b}}(\mathrm{mod}A)$ with a basic silting object $A$ on $\mathrm{K}^{\mathrm{b}}(\mathrm{proj}A)$. Assume that $|A|=n$ and $I$ is a subset of $\{1,2,\ldots,n\}$. Put $A_{I}=\bigoplus_{i\in I}A_{i}$ and $R_{I}=\bigoplus_{i\in I}R_{i}$, where each $A_{i}$ is indecomposable and $R_{i}$ is a complete set of pairwise non-isomorphic simple objects of $\mathrm{D}_{A}^{0}$. Define $\mathcal{X}_{R_{I}}=\{X\in\mathrm{D}_{A}^{0}~|~\mathrm{Hom}_{\mathcal{T}}(X,R_{I})=0\}$. Then we have the following corollary:

\begin{cor}\label{cor3} Let $A$ be a finite-dimensional $K$-algebra. Assume that $G\in\mathrm{D}^{\mathrm{b}}(\mathrm{mod}A)$, $|A|=n$, and $I$ is a subset of $\{1,2,\ldots,n\}$. Then we have the following commutative diagram:
\begin{center} $\xymatrix@C=25pt@R=20pt{
A\ar[d]^{\mathfrak{G}_{G}(\Psi(-))}\ar[r]^{\mu} &\mu_{A_{I}}^{L}(A)\ar[d]^{\mathfrak{G}_{G}(\Psi(-))}&\\
(\mathfrak{G}_{G}(\mathrm{D}^{\leq0}_{A}),\mathfrak{G}_{G}(\mathrm{D}^{\geq0}_{A}))\ar[r]^{T_{\mathrm{HRS}}^{\mathfrak{G}(\mathcal{X}_{R_{I}})}}& (\mathfrak{G}_{G}(\mathrm{D}^{\leq0}_{\mu_{A_{I}}^{L}(A)}),\mathfrak{G}_{G}(\mathrm{D}^{\geq0}_{\mu_{A_{I}}^{L}(A)})).&}$
\end{center}
\end{cor}
\begin{proof} This follows Corollary \ref{cor2}.
\end{proof}

Let $A$ be a dg $K$-algebra and let $\mathbf{C}_{\mathrm{dg}}(A)$ be the dg category of (right) dg $A$-modules, see
\cite[Section 3.1]{K}. The dg category $\mathbf{C}_{\mathrm{dg}}(A)$ is a pretriangulated dg category. Denote by $\mathbf{K}(A):=H^{0}\mathbf{C}_{\mathrm{dg}}(A)$ the homotopy category of $\mathbf{C}_{\mathrm{dg}}(A)$. Then $\mathbf{K}(A)$ is a triangulated category with shift functor $[1]$ induced by the shift of dg modules. The derived category $\mathbf{D}(A)$ of dg $A$-modules is defined as the triangle quotient of $\mathbf{K}(A)$ by the full subcategory of $\mathbf{K}(A)$ consisting of acyclic dg $A$-modules, see \cite[Section 2.2]{K}. We denote by $\mathbf{per}(A)$ the \emph{perfect derived category} that is the thick subcategory of $\mathbf{D}(A)$ generated by $A_{A}$ and by $\mathbf{D}_{\mathrm{fd}}(A)$ the \emph{finite-dimensional derived category} that is full subcategory of $\mathbf{D}(A)$ consisting of dg $A$-modules $M$ whose total cohomology is finite-dimensional over $K$.

Let $\mathbf{D}^{\leq0}$ $($resp., $\mathbf{D}^{\geq0})$ be the full subcategory of $\mathbf{D}(A)$ consisting of the dg $A$-modules $M$ with $H^{p}(M)=0$ for each $p>0$ $($resp., for each $p<0)$. Then $(\mathbf{D}^{\leq0},\mathbf{D}^{\geq0})$ is a $t$-structure on $\mathbf{D}(A)$ by \cite[Example 3.19]{AMY} with the associated truncation functors are the standard truncations and $H^{0}:~\mathbf{D}(A)\rightarrow \mathrm{Mod}H^{0}(A)$ restricts to an equivalence between the heart $\mathbf{D}^{0}$ and $\mathrm{Mod}H^{0}(A)$, see \cite[Section 2.1]{A}.

Let $\mathbf{D}^{-}_{\mathrm{fd}}(A)$ be the full subcategory of $\mathbf{D}_{\mathrm{fd}}(A)$ consisting of dg $A$-modules $M$ with $H^{p}(M)=0$ for each $p>>0$. Then $(\mathbf{D}^{\leq0},\mathbf{D}^{\geq0})$ induces a $t$-structure $(\mathbf{D}_{\mathrm{fd}}^{-,\leq0},\mathbf{D}_{\mathrm{fd}}^{-,\geq0})$ by \cite[3.1.19]{BBD}, where $\mathbf{D}_{\mathrm{fd}}^{-,\leq0}=\mathbf{D}^{\leq0}\bigcap\mathbf{D}^{-}_{\mathrm{fd}}(A)$ and $\mathbf{D}_{\mathrm{fd}}^{-,\geq0}=\mathbf{D}^{\geq0}\bigcap\mathbf{D}^{-}_{\mathrm{fd}}(A)$. We denote by $\mathbf{D}_{\mathrm{fd}}^{-,0}$ the heart of $(\mathbf{D}_{\mathrm{fd}}^{-,\leq0},\mathbf{D}_{\mathrm{fd}}^{-,\geq0})$.

\begin{cor}\label{cor3} Let $A$ be a dg $K$-algebra satisfying $H^{p}(A)=0$ for $p>0$ and $H^{p}(A)$ is finite-dimensional for any $p\in\mathbb{Z}$. Assume that $G\in\mathbf{D}_{\mathrm{fd}}(A)$, $|A|=n$ and $I$ is a subset of $\{1,2,\ldots,n\}$. Then we have the following commutative diagram:
\begin{center} $\xymatrix@C=25pt@R=20pt{
A\ar[d]^{\mathfrak{G}_{G}(\Psi(-))}\ar[r]^{\mu} &\mu_{A_{I}}^{L}(A)\ar[d]^{\mathfrak{G}_{G}(\Psi(-))}&\\
(\mathfrak{G}_{G}((\mathbf{D}_{\mathrm{fd}}^{-,\leq0})_{A}),\mathfrak{G}_{G}((\mathbf{D}_{\mathrm{fd}}^{-,\geq0})_{A}))\ar[r]^{T_{\mathrm{HRS}}^{\mathfrak{G}(\mathcal{X}_{S_{I}})}}& (\mathfrak{G}_{G}((\mathbf{D}_{\mathrm{fd}}^{-,\leq0})_{\mu_{A_{I}}^{L}(A)}),\mathfrak{G}_{G}((\mathbf{D}_{\mathrm{fd}}^{-,\geq0})_{\mu_{A_{I}}^{L}(A)})).&}$
\end{center}
where $A_{I}=\bigoplus_{i\in I}A_{i}$ with indecomposable $A_{i}$, $S_{I}=\bigoplus_{i\in I}S_{i}$ with a complete set $S_{i}$ of pairwise non-isomorphic simple objects of $\mathbf{D}_{\mathrm{fd}}^{-,0}$ and $\mathcal{X}_{S_{I}}=\{X\in\mathbf{D}_{\mathrm{fd}}^{-,0}~|~\mathrm{Hom}_{\mathcal{T}}(X,S_{I})=0\}$.
\end{cor}
\begin{proof} By \cite[Proposition 6.12]{AMY}, $(\mathbf{per}(A),\mathbf{D}_{\mathrm{fd}}(A))$ is an ST-pair on $\mathbf{D}^{-}_{\mathrm{fd}}(A)$ with a basic silting object $A$ on $\mathbf{per}(A)$ and ($(\mathbf{D}_{\mathrm{fd}}^{-})_{A}^{\leq0},(\mathbf{D}_{\mathrm{fd}}^{-})_{A}^{\geq0}$) is a bounded $t$-structure on $\mathbf{D}_{\mathrm{fd}}(A)$. So the result follows from Corollary \ref{cor2}.
\end{proof}

\bigskip
{\bf Acknowledgement.}
This work was partially supported by the National Natural Science Foundation of China (Grant No. 12361007).

\renewcommand\refname
\textbf{Jiaojiao Lu}\\
Department of Mathematics, Northwest Normal University, Lanzhou 730070, China.\\
E-mail: \textsf{lujiaojiao0218@163.com}\\[1mm]
\textbf{Zhongkui Liu}\\
Department of Mathematics, Northwest Normal University, Lanzhou 730070, China.\\
E-mail: \textsf{Liuzk@nwnu.edu.cn}\\[1mm]
\textbf{Renyu Zhao}\\
Department of Mathematics, Northwest Normal University, Lanzhou 730070, China.\\
E-mail: \textsf{zhaory@nwnu.edu.cn}\\[1mm]
\end{document}